\documentclass[smallextended]{svjour3}
\usepackage{amsmath, amsfonts, amssymb, latexsym}
\usepackage{tabularx, longtable, makecell, arydshln, float}
\usepackage{multirow, multicol, array, placeins}
\usepackage[mathscr]{eucal}
\usepackage[normalem]{ulem}
\usepackage{graphicx, xcolor}
\usepackage{algpseudocode, algorithm, comment}
\usepackage[numbers]{natbib}
\allowdisplaybreaks
\usepackage{enumitem}
\usepackage[justification=centering]{caption}

\newcommand{\cH}{\mathcal{H}}

\newcommand{\R}{\mathbb{R}}

\newtheorem{defi}{Definition}

\newtheorem{thm}{Theorem}

\newtheorem{lem}{Lemma}
\newtheorem{rem}{Remark}

\def\beq{\begin{equation}}
\def\eeq{\end{equation}}
\def\beqs{\begin{equation*}}
\def\eeqs{\end{equation*}}

\newcommand{\todo}[1]{\textcolor{red}{#1}}

\begin{document}

\title{Splitting gradient algorithms for solving monotone equilibrium problems
\thanks{The authors were supported in part by National Foundation 
for Science and Technology Development (NAFOSTED, Vietnam) 
under grant number 101.01-2017.315.}
}

\author{Le Dung Muu \and Phung Minh Duc \and  Xuan Thanh Le}
\institute{Le Dung Muu \at
Thang Long Institute of Mathematics and Applied Sciences (TIMAS), Thang Long University, Hanoi, Vietnam,
\email{ldmuu@math.ac.vn} \and
Phung Minh Duc \at  Faculty of Basic Science, College of Technology Medical Equipment, Hanoi, Vietnam,\\
\email{ducphungminh@gmail.com} \and
Xuan Thanh Le \at
Institute of Mathematics, Vietnam Academy of Science and Technology, Hanoi, Vietnam,\\
\email{lxthanh@math.ac.vn}
}

\maketitle

\begin{abstract}
It is well known that the projection method is not convergent for monotone equilibrium problems. 
Recently Sosa \textit{et al.} in \cite{SS2011} proposed a projection algorithm 
ensuring convergence for paramonotone equilibrium problems. 
In this paper we modify 
\todo{this algorithm}
to obtain a splitting convergent one for the case when the bifunction is the sum of the two ones. 
At each iteration, two strongly convex subprograms are required to solve separately, 
one for each component bifunction. We show that the algorithm is convergent for paramonotone bifunction 
without any Lipschitz type condition as well as H\"older continuity of the \todo{involved bifunctions}. 
Furthermore, we show that the ergodic sequence defined by the \todo{algorithm's iterates} 
converges to a solution without paramonotonicity property.
We use the proposed algorithm to solve a jointly constrained Cournot-Nash model.
The computational results show that this algorithm is efficient for the model with a restart strategy.


\end{abstract}


\section{Introduction}

Let $\cH$ be a real Hilbert space endowed with weak topology defined by the inner product
$\langle \cdot , \cdot \rangle$ and its induced norm $\| \cdot \|$. Let $C \subseteq \cH$ be a
nonempty closed convex subset and $f: \cH \times \cH \to \R \cup \{+\infty\}$ a bifunction
such that $f(x, y) < +\infty$ for every $x, y \in C$. The equilibrium problem defined by the
Nikaido-Isoda-Fan inequality that we are going to consider in this paper is given as
   $$\text{Find}\ x \in C: f(x, y) \geq 0 \  \forall y \in C. \eqno(EP)$$
This inequality first was used in 1955 by Nikaido-Isoda \cite{NI1955} in convex game models.
Then in 1972 Ky Fan \cite{F1972} called this inequality a minimax one and established existence
theorems for Problem $(EP)$. After the appearance of the paper by Blum and Oettli \cite{BO1994}
  Problem $(EP)$ has been contracted much attention
of researchers. It has been shown in \cite{BCPP2013, BO1994, MO1992} that some important problems
such as optimization, variational inequality, Kakutani fixed point and Nash equilibrium can be
formulated in the form of $(EP)$. Many papers concerning the solution existence, stabilities
as well as algorithms for Problem $(EP)$ have been published
(see e.g. \cite{HM2011, IS2003, M2003, MQ2009, QAM2012, QMH2008, SS2011}
and the   survey paper \cite{BCPP2013}). A basic method for Problem $(EP)$ is the gradient
(or projection) one, where the sequence of iterates is defined by taking
\begin{equation}\label{1m}
 x^{k+1} = \min\left\{ \lambda_k f(x^k,y) +\frac{1}{2} \|y-x^k\|^2 : y \in C \right\},
\end{equation}
with $\lambda_k$ is some appropriately chosen real number.
Note that in the variational inequality case, where $f(x,y) := \langle F(x), y-x\rangle$,
the iterate $x^{k+1}$ defined by
(\ref{1m}) becomes
$$x^{k+1} = P_C\left(x^k - \lambda_k F(x^k)\right),$$
where $P_C$ stands for the metric projection onto $C$.

It is well known that  under certain conditions on the parameter $\lambda_k$,
the projection method is convergent if $f$ is strongly pseudomonotone or
paramonotone \cite{IS2003, QMH2008}. However when $f$ is monotone, it may fail to converge.
In order to obtain convergent algorithms for monotone, even pseudomonotone, equilibrium problems,
the extragradient method first proposed by Korpelevich \cite{K1976} for the saddle point
and related problems has been extended to equilibrium problems \cite{QMH2008}.
In this extragradient algorithm, at each iteration, it requires solving the two strongly
convex programs
 \begin{equation}\label{2m}
 y^k = \min\left\{ \lambda_k f(x^k,y) +\frac{1}{2} \|y-x^k\|^2 : y \in C \right\},
\end{equation}
 \begin{equation}\label{3m}
 x^{k+1} = \min\left\{ \lambda_k f(x^k,y) +\frac{1}{2} \|y-y^k\|^2 : y \in C \right\},
\end{equation}
which may cause computational cost. In order to reduce the computational cost,
several convergent algorithms that require solving only one strongly convex program
or computing only one projection at each iteration have been proposed.
\todo{These algorithms were applied to} some classes of bifunctions 
\todo{such as} strongly pseudomonotone and paramonotone \todo{ones},
with or without using an ergodic sequence (see e.g. \cite{AHT2016, DMQ2016, SS2011}). 
In another direction, also for the sake of reducing computational cost, 
some splitting algorithms have been developed
(see e.g. \cite{AH2017, HV2017, M2009}) for monotone equilibrium problems where
the bifunctions can be decomposed into the sum of two bifunctions.
In these algorithms the convex subprograms (resp. regularized subproblems) involving
the bifunction $f$ can be replaced by the two convex subprograms (resp. regularized subproblems),
one for each $f_i$ $(i=1, 2)$ independently.

In this paper we modify the projection algorithm in \cite{SS2011} 
to obtain  a splitting convergent algorithm for paramonotone equilibrium problems. 
The main future of this algorithm is that
at each iteration, it requires solving only one strongly convex
program. Furthermore, in the case when
$f = f_1 + f_2$, this strongly convex subprogram can be replaced by the two strongly convex
subprograms, one for each $f_1$ and $f_2$ as the algorithm in \cite{AH2017, HV2017},
but for the convergence we do not require any additional conditions such as
H\"older continuity and Lipschitz type condition as in \cite{AH2017, HV2017}. 
We also show that the ergodic sequence defined by the \todo{algorithm's} iterates 
\todo{converges} to a solution without paramonotonicity property.
We apply the ergodic algorithm for solving a Cournot-Nash model with joint constraints. 
The computational results and experiences show that the ergodic algorithm is efficient for this model with a restart strategy.

The remaining part of the paper is organized as follows. The next section \todo{lists} preliminaries
containing some lemmas that will be used in proving the convergence of  the proposed algorithm.
Section \ref{SectionAlgorithm} is devoted to the description of
the algorithm and its convergence analysis. 
\todo{Section \ref{SectionExperiments} shows an application of the algorithm 
in solving a Cournot-Nash model with joint constraints.
Section \ref{SectionConclusion} closed the paper with some conclusions.
}


\section{Preliminaries}

We recall from \cite{BCPP2013} the following well-known definition on monotonicity of bifunctions.
\begin{defi}
A bifunction $f: \cH \times \cH \to \R \cup \{+\infty\}$ is said to be
\begin{itemize}[leftmargin = 0.5 in]
\item[(i)] strongly monotone on $C$ with modulus $\beta > 0$ (shortly $\beta$-strongly monotone) if
  $$f(x, y) + f(y, x) \leq -\beta \| y - x \|^2  \quad \forall x, y \in C;$$

\item[(ii)] monotone on $C$ if
  $$f(x, y) + f(y, x) \leq 0 \quad \forall x, y \in C;$$

\item[(iii)] strongly pseudo-monotone on $C$ with modulus $\beta > 0$
(shortly $\beta$-strongly pseudo-monotone) if
  $$f(x, y) \geq 0 \implies f(y, x) \leq -\beta\| y - x \|^2 \quad \forall x, y \in C;$$

\item[(iv)] pseudo-monotone on $C$ if
  $$f(x, y) \geq 0 \implies f(y, x) \leq 0 \quad \forall x, y \in C.$$

\item[(v)] paramonotone on $C$ with respect to a set $S$ if
 $$x^* \in S, x\in C \text{ and } f(x^*, x) = f(x, x^*) = 0 \text{ implies } x \in S.$$
\end{itemize}
\end{defi}

Obviously, $(i) \implies (ii) \implies (iv)$ and $(i) \implies (iii) \implies (iv)$.
Note that a strongly pseudo-monotone bifunction may not be monotone.
Paramonotone bifunctions have been used in e.g. \cite{SS2011,S2011}. Some properties of paramonotone operators can be found in \cite{IS1998}.
Clearly in the case of optimization problem when
$f(x,y) = \varphi(y) - \varphi(x)$, the bifunction  $f$ is paramonotone on $C$ with
respect to the solution set of the problem $\min_{x\in C} \varphi(x)$.

The following well known lemmas will be used for proving the convergence of the algorithm
to be described in the next section.


\begin{lem}\label{lem1}{\em (see \cite{TX1993} Lemma 1)}
Let $\{\alpha_k\}$ and $\{\sigma_k\}$ be two sequences of nonnegative numbers such that
$\alpha_{k+1} \leq \alpha_k + \sigma_k$ for all $k \in \mathbb{N}$,
where $\sum_{k=1}^{\infty} \sigma_k < \infty$. Then the sequence $\{\alpha_k\}$ is convergent.
\end{lem}

\begin{lem}\label{lem2}{\em (see \cite{P1979})}
Let $\cH$ be a Hilbert space, $\{x^k\}$ a sequence in $\cH$. Let $\{r_k\}$ be a sequence of
nonnegative number such that $\sum_{k=1}^{\infty} r_k = +\infty$ and set
$z^k := \dfrac{\sum_{i=1}^k r_i x^i}{\sum_{i=1}^k r_i}$.
Assume that there exists a nonempty, closed convex set $S \subset \cH$ satisfying:
\begin{itemize}[leftmargin = 0.5 in]
\item[(i)] For every $z \in S$, $\lim_{n \to \infty}\|z^k - z\|$ exists;
\item[(ii)] Any weakly cluster point of the sequence $\{z^k\}$ belongs to $S$.
\end{itemize}
Then the sequence $\{z^k\}$ weakly converges.
\end{lem}

\begin{lem}\label{LemmaXu} {\em (see \cite{X2002})}
Let $\{\lambda_k\}, \{\delta_k\}, \{\sigma_k\}$ be sequences of real numbers such that
\begin{itemize}[leftmargin = 0.5 in]
 \item[(i)] $\lambda_k \in (0, 1)$ for all $k \in \mathbb{N}$;
 \item[(ii)] $\sum_{k = 1}^{\infty} \lambda_k = +\infty$;
 \item[(iii)] $\limsup_{k \to +\infty} \delta_k \le 0$;
 \item[(iv)] $\sum_{k = 1}^{\infty} |\sigma_k| < +\infty$.
\end{itemize}
Suppose that $\{\alpha_k\}$ is a sequence of nonnegative real numbers satisfying
\beqs
\alpha_{k+1} \le (1 - \lambda_k) \alpha_k + \lambda_k \delta_k + \sigma_k \qquad \forall k \in \mathbb{N}.
\eeqs
Then we have $\lim_{k \to +\infty} \alpha_k = 0$.
\end{lem}


\section{The problem, algorithm and its convergence}\label{SectionAlgorithm}

\subsection{The problem}
In what follows, for the following equilibrium problem
$$\text{Find}\ x \in C: f(x, y) \geq 0 \  \forall y \in C \eqno(EP)$$
  we suppose that $f(x, y) = f_1(x, y) + f_2(x, y)$
and that $f_i(x, x) = 0$ ($i=1, 2$) for every $x, y \in C$.
The following assumptions for the bifunctions $f, f_1, f_2$ will be used in the sequel.

\begin{itemize}[leftmargin = 0.5 in]
\item[(A1)] For each $i =1, 2$ and each $x\in C$,
the function $f_i(x, \cdot)$ is convex and sub-differentiable,
while $f(\cdot, y)$ is weakly upper semicontinuous on $C$;
\item[(A2)] If $\{x^k\} \subset C$ is bounded,
then for each $i = 1, 2$, the sequence $\{g^k_i\}$
with $g^k_i \in \partial f_i(x^k,x^k)$ is bounded;
\item[(A3)] The bifunction $f$ is monotone on $C$.
\end{itemize}

Assumption (A2) has been used in e.g. \cite{S2011}.
Note that Assumption (A2) is satisfied if the functions $f_1$ and $f_2$ are jointly weakly continuous
on an open convex set containing $C$ (see \cite{VSN2015} Proposition 4.1).\\

The dual problem of $(EP)$ is
$$\text{find}\ x \in C: f(y, x) \leq 0 \  \forall y \in C. \eqno(DEP)$$
We  denote the solution sets of $(EP)$ and $(DEP)$ by $S(C,f)$ and $S^d(C,f)$, respectively.
A relationship between $S(C,f)$ and $S^d(C,f)$ is given in the following lemma.

\begin{lem}\label{lem3}{\em (see \cite{KS2000} Proposition 2.1)}
(i) If $f(\cdot, y)$ is weakly upper semicontinuous
and $f(x, \cdot)$ is convex for all $x, y \in C$,
then $S^d(C,f) \subset S(C,f)$.

(ii) If $f$ is pseudomonotone, then $S(C, f) \subset S^d(C, f)$.
\end{lem}

Therefore, under the assumptions (A1)-(A3) one has $S(C, f) = S^d(C, f)$. In this paper
we suppose that $S(C, f)$ is nonempty.


\subsection{The algorithm and its convergence analysis}

The algorithm below is a gradient one for paramonotone equilibrium problem $(EP)$.
The stepsize is computed as in the algorithm for equilibrium problem in \cite{SS2011}.

\begin{algorithm}[H]
\caption{A splitting algorithm for solving paramonotone or strongly
pseudo-monotone equilibrium problems.}\label{alg2}
\begin{algorithmic}
\State \textbf{Initialization:} Seek $x^0\in C$.
Choose a sequence $\{\beta_k\}_{k \geq 0} \subset \mathbb{R}$
satisfying the following conditions
\beqs
\quad \sum_{k=0}^\infty \beta_k = +\infty, \quad \sum_{k=0}^\infty \beta_k^2 < +\infty.
\eeqs
\State \textbf{Iteration} $k = 0, 1, \ldots$:\\
\qquad Take $g_1^k \in \partial_2 f_1(x^k, x^k), g_2^k \in \partial_2 f_2(x^k, x^k)$.\\
\qquad Compute
\begin{align*}
\eta_k &:= \max\{\beta_k, \|g_1^k\|, \|g_2^k\|\}, \ \lambda_k := \dfrac{\beta_k}{\eta_k},\\
y^k &:= \arg\min\{\lambda_k f_1(x^k, y) + \dfrac{1}{2}\|y - x^k\|^2 \mid y \in C\},\\
x^{k+1} &:= \arg\min\{\lambda_k f_2(x^k, y) +\dfrac{1}{2}\|y - y^k\|^2 \mid y \in C\}.
\end{align*}
\end{algorithmic}
\end{algorithm}


\begin{thm} \label{thm1}
In addition to the assumptions {\em (A1), (A2), (A3)} we suppose that
$f$ is paramonotone on $C$, and that either int $C \not=\emptyset$ or for each $x\in C$ both bifunctions $f_1(x, \cdot)$, $f_2(x, \cdot)$ are continuous at a point in $C$. Then the sequence $\{x^k\}$ generated
by  the algorithm \ref{alg2} converges weakly to a solution of $(EP)$.
Moreover, if $f$ is strongly pseudomonotone, then $\{x^k\}$ strongly converges to
the unique solution of $(EP)$.
\end{thm}
\begin{proof} First, we show that, for each $x^* \in S(f, C)$, the sequence $\{\|x^k - x^*\|\}$ is convergent.

Indeed, for each $k \geq 0$, for simplicity of notation, let
\begin{equation*}
h_1^k (x) := \lambda_k f_1(x^k, x) + \frac{1}{2}\| x-x^k \|^2,
\end{equation*}
\begin{equation*}
h_2^k (x) := \lambda_k f_2(x^k, x) + \frac{1}{2}\| x-y^k \|^2.
\end{equation*}
By Assumption (A1), the functions $h_1^k$ is strongly convex with modulus $1$
and subdifferentiable, which implies
\begin{equation} \label{ct2}
h_1^k (y^k) + \langle u_1^k, x - y^k \rangle + \frac{1}{2} \| x - y^k \|^2 \leq h_1^k(x) \quad \forall x \in C
\end{equation}
for any $u_1^k \in \partial h_1^k (y^k)$. On the other hand, from the definition of $y^k$, 
using the regularity condition, by the optimality condition for convex programming, we have
$$0 \in \partial h_1^k (y^k) + N_C (y^k)$$
 In turn, this implies that there exists
$-u_1^k \in \partial h_1^k (y^k)$ such that $\langle u_1^k, x - y^k \rangle \geq 0$
for all $x \in C$. Hence, from (\ref{ct2}), for each $x \in C$, it follows that
$$ h_1^k(y^k) + \frac{1}{2} \| x - y^k \|^2 \leq h_1^k(x), $$
i.e.,
$$\lambda_k f_1(x^k, y^k) + \frac{1}{2}\| y^k - x^k \|^2 + \dfrac{1}{2}\|x - y^k\|^2
\leq \lambda_k f_1(x^k, x) + \frac{1}{2}\| x - x^k \|^2,$$
or equivalently,
\begin{equation} \label{ct3}
\|y^k - x\|^2 \leq \|x^k - x\|^2 +2\lambda_k \left( f_1(x^k, x)-f_1(x^k, y^k) \right) - \|y^k - x^k\|^2.
\end{equation}
Using the same argument for $x^{k+1}$, we obtain
\begin{equation} \label{ct4}
\|x^{k+1} - x\|^2 \leq \|y^k - x\|^2 +2\lambda_k \left( f_2(x^k, x) -f_2(x^k, x^{k+1}) \right) - \|x^{k+1} - y^k\|^2.
\end{equation}
Combining (\ref{ct3}) and (\ref{ct4}) yields
\begin{align} \label{ct5}
 \|x^{k+1} - x\|^2 &\leq \|x^k - x\|^2 - \|y^k - x^k\|^2 - \|x^{k+1} - y^k\|^2 \notag\\
 & \quad + 2 \lambda_k \left( f_1(x^k, x) + f_2(x^k, x) \right) -2 \lambda_k \left( f_1(x^k, y^k) + f_2(x^k, x^{k+1}) \right) \notag\\
&= \|x^k - x\|^2 - \|y^k - x^k\|^2 - \|x^{k+1} - y^k\|^2 \notag\\
& \quad + 2 \lambda_k f(x^k, x) - 2\lambda_k \left( f_1(x^k, y^k) + f_2(x^k, x^{k+1}) \right).
\end{align}
From $g_1^k \in \partial_2f_1(x^k,x^k)$ and $f_1(x^k, x^k) = 0$, it follows that
\beqs
f_1(x^k, y^k) - f_1(x^k, x^k) \ge \langle g_1^k, y^k - x^k \rangle,
\eeqs
which implies
\beq \label{ct6}
-2 \lambda_k f_1(x^k, y^k) \leq - 2 \lambda_k \langle g_1^k, y^k - x^k \rangle.
\eeq
By using the Cauchy-Schwarz inequality and the fact that $\|g_1^k\| \leq \eta_k$, from (\ref{ct6})
one can write
\beq \label{ct7}
-2 \lambda_k f_1(x^k, y^k) \leq 2 \frac{\beta_k}{\eta_k} \eta_k \|y^k - x^k\| = 2 \beta_k \|y^k - x^k\|.
\eeq
By the same argument, we obtain
\beq \label{ct8}
-2 \lambda_k f_2(x^k, x^{k+1}) \leq 2 \beta_k \|x^{k+1} - x^k\|.
\eeq
Replacing (\ref{ct7}) and (\ref{ct8}) to (\ref{ct5}) we get
\begin{align} \label{ct9}
 \|x^{k+1} - x\|^2 &\leq \|x^k - x\|^2 + 2\lambda_k f(x^k, x) \notag\\
 &\quad + 2 \beta_k \|y^k - x^k\| + 2 \beta_k \|x^{k+1}-x^k\| - \|y^k - x^k\|^2 - \|x^{k+1} - y^k\|^2 \notag\\
 &= \|x^k - x\|^2 + 2\lambda_k f(x^k, x) \notag \\
 &\quad + 2\beta_k^2 - \left(\|y^k - x^k\| - \beta_k\right)^2 - \left(\|x^{k+1} - x^k\| - \beta_k\right)^2 \notag\\
 &\leq  \|x^k - x\|^2 + 2\lambda_k f(x^k, x) + 2\beta_k^2.
\end{align}
Note that by definition of $x^* \in S(f, C) = S^d(f, C)$ we have $f(x^k, x^*) \leq 0$.
Therefore, by taking $x = x^*$ in (\ref{ct9}) we obtain
\begin{align}\label{ct12}
\|x^{k+1} - x^*\|^2 &\leq \|x^k - x^*\|^2 + 2 \lambda_k f(x^k, x^*) + 2\beta_k^2 \notag\\
&\leq \|x^k - x^*\|^2 +  2\beta_k^2.
\end{align}
Since $\sum_{k = 0}^{\infty} \beta_k^2 < +\infty$ by assumption, in virtue of Lemma
\ref{lem1}, it follows from (\ref{ct12}) that the sequence $\{\|x^k - x^*\|\}$ is convergent.

Next, we prove that any cluster point of the sequence $\{x^k\}$ is a solution of $(EP)$.

Indeed, from  (\ref{ct12}) we have
\beq\label{ct16bc}
- 2 \lambda_k f(x^k, x^*) \leq \|x^k - x^*\|^2 - \|x^{k+1} - x^*\|^2 + 2 \beta_k^2 \quad \forall k \in \mathbb{N}.
\eeq
By summing up we obtain
\beqs
2 \displaystyle \sum_{i = 0}^\infty \lambda_i\left(- f(x^i, x^*)\right)
\leq \|x^0 - x^*\|^2  + 2 \displaystyle \sum_{i = 0}^\infty \beta_i^2 < \infty.
\eeqs

On the other hand, by  Assumption (A2) the sequences $\{g_1^k\}, \{g_2^k\}$ are bounded.
This fact, together with the construction of $\{\beta_k\}$, implies that
there exists $M > 0$ such that $\|g_1^k\| \leq M, \|g_2^k\| \leq M, \beta_k \leq M$
for all $k \in \mathbb{N}$. Hence for each $k \in \mathbb{N}$ we have
\beqs
\eta_k = \max\{\beta_k, \|g_1^k\|, \|g_2^k\|\} \leq M,
\eeqs
which implies
  $\sum_{i=0}^\infty \lambda_i = \infty$.  Thus, from  $f(x^i,x^*) \leq 0$,  it holds that
$$\lim\sup f(x^k,x^*) = 0 \quad \forall x^* \in S(C,f).$$
Fixed $x^* \in S(C,f)$ and let $\{x^{k_j}\}$ be a subsequence of $\{x^k\}$ such that
$$\lim_{k}\sup f(x^k,x^*) = \lim_{j} f(x^{k_j},x^*) = 0.$$
Since $\{x^{k_j}\}$ is bounded, we may assume that $\{x^{k_j}\}$ weakly converges to some $\bar{x}$.
Since $f(\cdot,x^*)$ is weakly upper semicontinuous by assumption (A1), we have
\beq\label{ct21}
f(\bar{x},x^*) \geq \lim f(x^{k_j},x^*) = 0.
\eeq
Then it follows from the monotonicity of $f$ that
$f(x^*,\bar{x}) \leq 0$. On the other hand, since $x^* \in S(C, f)$, by definition we have
$f(x^*,\bar{x}) \geq 0$. Therefore we obtain $f(x^*,\bar{x}) = 0$.
Again, the monotonicity of $f$  implies $f(\bar{x}, x^*) \le 0$, and therefore,
by (\ref{ct21})  one has $f(\bar{x}, x^*) = 0$.
Since $f(x^*,\bar{x}) = 0$ and $f(\bar{x}, x^*) = 0$, it follows from paramonotonicity of $f$
that $\bar{x}$ is a solution to $(EP)$. Since $\|x^k - \bar{x}\|$ converges,
from the fact that $x^{k_j}$ weakly converges to $\bar{x}$, we can conclude that
the whole sequence $\{x^k\}$ weakly converges to $\bar{x}$.

 Note that if $f$ is strongly pseudomonotone,
then Problem $(EP)$ has a unique solution (see \cite{MQ2015} Proposition 1).
Let $x^*$ be the unique solution of $(EP)$. By definition of $x^*$ we have
\beqs
f(x^*, x) \ge 0 \quad \forall x \in C,
\eeqs
which, by strong pseudomonotonicity of $f$, implies
\beq\label{ct22}
f(x, x^*) \le - \beta \|x-x^*\|^2 \quad \forall x \in C.
\eeq
By choosing $x = x^k$ in (\ref{ct22}) and then applying to (\ref{ct9}) we obtain
\beqs
\|x^{k+1} - x^*\|^2 \le (1 - 2\beta\lambda_k) \|x^k - x^*\|^2 + 2 \beta_k^2 \quad \forall k \in \mathbb{N},
\eeqs
 which together with   the construction of $\beta_k$ and $\lambda_k$, by virtue of
Lemma \ref{LemmaXu} with $\delta_k \equiv 0$,  implies  that
\beqs
\lim_{k \to +\infty} \|x^k - x^*\|^2 = 0,
\eeqs
i.e., $x^k$ strongly converges to the unique solution $x^*$ of $(EP)$.
$\square$
\end{proof}

The following simple example shows  that without paramonotonicity, the algorithm may not be convergent.
Let us consider the following example, taken from \cite{FP2003}, where
$f_(x,y):= \langle Ax, y-x\rangle$ and $C:= \mathbb{R}^2$ and
\beqs
A = \begin{bmatrix}
0 & 1\\
-1 & 0
\end{bmatrix}.
\eeqs
Clearly,  $x^*=(0,0)^T$ is the unique solution of this problem.
It is easy to check that this bifunction is monotone, but not paramonotone.
An elementary computation shows that
$$x^{k+1} = x^k - \lambda_k A x^k = (x^k_1-\lambda_k x^k_2, x^k_2 + \lambda_k x^k_1)^T.$$
Thus, $\|x^{k+1}\|^2 = (1+\lambda^2_k)\|x^k\|^2 > \|x^k\|^2$ if $x^k \neq 0$ ,
which implies that the sequence $\{x^k\}$  does not converge to the solution $x^* = 0$
for any starting point $x^0 \neq 0$.

To illustrate  our motivation let us consider the following optimization problem
\begin{align*}
(OP) \quad \min \quad &\varphi(x) := \frac{1}{2} x^T Q x - \displaystyle \sum_{i=1}^n \ln (1 + \max\{0, x_i\})\\
\text{subject to} \quad &x_i \in [a_i, b_i] \subset \mathbb{R} \quad (i = 1, \ldots, n),
\end{align*}
where $Q \in \mathbb{R}^{n \times n}$ is a positive semidefinite matrix.
This problem is equivalent to the following equilibrium problem
\beqs
\text{Find } x^* \in C \text{ such that } f(x^*, y) \ge 0 \ \forall y \in C,
\eeqs
where $C := [a_1, b_1] \times \ldots \times [a_n, b_n]$, and $f(x, y) := \varphi(y) - \varphi(x)$.
We can split the function
$f(x, y) = f_1(x, y) + f_2(x, y)$
by taking
\beqs
f_1(x, y) = \frac{1}{2} y^T Q y - \frac{1}{2} x^T Q x,
\eeqs
and
\beqs
f_2(x, y) = \displaystyle \sum_{i = 1}^n \left( \ln(1 + \max\{0, x_i\}) - \ln(1 + \max\{0, y_i\}) \right).
\eeqs
Since $Q$ is a positive semidefinite matrix and $\ln(\cdot)$ is concave on $(0, +\infty)$, the functions $f_1, f_2$
are equilibrium functions satisfying conditions (A1)-(A3).
Clearly, $f_1(x, \cdot)$ is convex quadratic, not necessarily separable,
while  $f_2(x, \cdot)$ is separable, not necessarily differentiable,
but their sum does not inherit these properties.

In order to obtain the convergence without paramonotonicity we  use the iterate $x^k$ to define an ergodic sequence by tanking
$$z^k:= \dfrac{\sum_{i=0}^k \lambda_i x^i} { \sum_{i=0}^k \lambda_i}.$$
Then we have the following convergence result.

\begin{thm}\label{thm2} Under the assumption in Theorem 1,  the ergodic  sequence $\{z^k\}$
  converges weakly to a solution of $(EP)$.
\end{thm}
  \begin{proof} In the proof of Theorem 1, we have shown that   the sequence $\{\|x^k - x^*\|\}$ is convergent.
   By the definition of $z^k$, the sequence $\{\|z^k - x^*\|\}$ is convergent too. In order to apply Lemma  \ref{lem2}, now we show that all weakly cluster points
   of $\{z^k\}$ belong to $S(f, C)$.
  In fact,  using the inequality (\ref{ct12}),  by taking the sum of its two sides over all indices
we have
\begin{align*}
2 \displaystyle \sum_{i = 0}^k \lambda_i f(x, x^i)
&\leq \displaystyle \sum_{i = 0}^k \left(\|x^i - x\|^2  - \|x^{i+1} - x\|^2 + 2 \beta_i^2\right)\\
&= \|x^0 - x\|^2  - \|x^{k+1} - x\|^2 + 2 \displaystyle \sum_{i = 0}^k \beta_i^2\\
&\leq \|x^0 - x\|^2  + 2 \displaystyle \sum_{i = 0}^k \beta_i^2.
\end{align*}
By using this inequality,  from definition of $z^k$ and convexity of $f(x, \cdot)$, we  can write
\begin{align}\label{ct18}
f(x, z^k) &= f\left(x, \dfrac{\sum_{i=0}^k \lambda_i x^i}{\sum_{i=0}^k \lambda_i} \right) \notag\\
&\leq \dfrac{\sum_{i = 0}^k \lambda_i f(x, x^i)}{ \sum_{i = 0}^k \lambda_i} \notag\\
&\leq \dfrac{\|x^0 - x\|^2  + 2 \sum_{i = 0}^k \beta_i^2}{2  \sum_{i = 0}^k \lambda_i}.
\end{align}

As we have shown in the proof of Theorem 1 that
\beqs
\lambda_k = \dfrac{\beta_k}{\eta_k} \geq \dfrac{\beta_k}{M}.
\eeqs
Since $\sum_{k = 0}^{\infty} \beta_k = +\infty$, we  have
 $\sum_{k=0}^{\infty} \lambda_k = +\infty.$
Then, it follows from
 (\ref{ct18})  that
\beq\label{ct20}
\lim_{k \to \infty} \inf f(x, z^k) \leq 0.
\eeq
Let $\bar{z}$ be any weakly cluster of $\{z^k\}$. Then there exists a subsequence $\{z^{k_j}\}$ of
$\{z^k\}$ such that $z^{k_j} \rightharpoonup \bar{z}$. 
Since $f(x, \cdot)$ is lower semicontinuous, tt follows from (\ref{ct20}) that
\beqs
f(x, \bar{z}) \le 0.
\eeqs
Since this inequality hold for arbitrary $x \in C$, it means that $\bar{z} \in S^d(f, C) = S(f, C)$.
 Thus it follows from  Lemma  \ref{lem2} that the sequence $\{z^k\}$ converges weakly
to a point $z^* \in S(f, C)$, which is a solution to $(EP)$.
$\square$
\end{proof}

\begin{rem}\label{rema2}
In case that $\cH$ is of finite dimension, we have
$\|z^{k+1}-z^k\| \to 0$ as $k\to \infty$. Since $\sum_{k \to +\infty} \lambda_k^2 < +\infty$,
at large enough iteration $k$, the value of $\lambda_k$
closes to $0$, which makes the intermediate iteration points $y^k, x^{k+1}$ close to $x^k$.
In turn, the new generated ergodic point $z^{k+1}$ does not change much from the previous one.
This slows down the convergence of the sequence $\{z^k\}$. In order to enhance
the convergence of the algorithm, it suggests a restart strategy
by replacing the starting point $x^0$ with $x^k$
whenever $\|z^{k+1}-z^k\| \leq \tau$ with an appropriate $\tau > 0$.
\end{rem}


\section{Numerical experiments}\label{SectionExperiments}

We used MATLAB R2016a for implementing  the proposed algorithms. All experiments were
conducted on a computer with a Core i5 processor, 16 GB of RAM, and Windows 10.

As we have noted in Remark \ref{rema2}, to improve the performance of our
proposed algorithm, we reset $x^0$ to $x^k$ whenever $\|z^{k+1}-z^k\| \leq \tau$
with an appropriate $\tau > 0$ and then restart the algorithm from beginning with
the new starting point $x^0$ if the stoping criterion $\|z^{k+1}-z^k\| \leq \epsilon$ is still not satisfied.  In all   experiments, we set $\tau := 10^{-3}$,
and terminated the algorithm when either the number of iterations exceeds $10^4$,
or the distance between the two consecutive generated ergodic points is less than $\epsilon : =10^{-4}$
(i.e., $\|z^{k+1} - z^k\| < 10^{-4}$). All the tests reported below were solved within 60 seconds.

We applied Algorithm 1 to compute  a Nash equilibrium of  a  linear Cournot oligopolistic model
with some additional joint constraints on the model's variables. The precise description of this model
is as follows.

There are $n$ firms producing a common homogeneous commodity.
Let $x_i$ be the production level of firm $i$, and $x = (x_1, \ldots, x_n)$ the vector of
production levels of all these firms. Assume that the production price $p_i$
given by firm $i$ depends on the total quantity $\sigma = \sum_{i = 1}^n x_i$
of the commodity as follows
\begin{center}
$p_i(\sigma) = \alpha_i - \delta_i \sigma\qquad(\alpha_i > 0, \delta_i > 0, i = 1, \ldots, n).$
\end{center}
Let $h_i(x_i)$ denote the production cost of firm $i$ when its production
level is $x_i$ and assume that the cost functions are affine of the forms
\begin{center}
$h_i(x_i) = \mu_i x_i + \xi_i \qquad(\mu_i > 0, \xi_i \ge 0, i = 1, \ldots, n).$
\end{center}
The profit of firm $i$ is then given by
\begin{center}
$q_i(x_1, \ldots, x_n) = x_i p_i(x_1 + \ldots + x_n) - h_i(x_i) \qquad (i = 1, \ldots, n).$
\end{center}
Each firm $i$ has a strategy set $C_i \subset \mathbb{R}_+$
consisting of its possible production levels, i.e., $x_i \in C_i$.
Assume that there are lower and upper bounds on quota of the commodity
(i.e., there exist $\underline{\sigma}, \overline{\sigma} \in \mathbb{R}_+$ such that
$\underline{\sigma} \le \sigma = \sum_{i = 1}^n x_i \le \overline{\sigma}$).
So the set of feasible production levels can be described by
\beqs
\Omega := \{x \in \mathbb{R}^n_+ \ | \ x_i \in C_i (i = 1, \ldots, n), \sum_{i = 1}^n x_i \in [\underline{\sigma}, \overline{\sigma}]\}.
\eeqs
Each firm $i$ seeks to maximize its profit by choosing the corresponding
production level $x_i$ under the presumption that the production
of the other firms are parametric input. In this context, a Nash equilibrium point for
the model is a point $x^* \in \Omega$ satisfying
\begin{center}
$q_i(x^*[x_i]) \le q_i(x^*) \qquad \forall x \in \Omega, i = 1,\ldots, n,$
\end{center}
where $x^*([x_i])$ stands for the vector obtained from $x^*$ by replacing
the component $x_i^*$ by $x_i$. It means that, if some firm $i$ leaves
its equilibrium strategy while the others keep their equilibrium positions,
then the profit of firm $i$ does not increase.
It has been shown that the unique Nash equilibrium point $x^*$ is also
the unique solution to the following equilibrium problem
\begin{equation}
\textrm {Find $x \in \Omega$ such that $f(x, y) := (\tilde{B} x + \mu - \alpha)^T (y - x) + \frac{1}{2} y^T B y - \frac{1}{2} x^T B x \ge 0 \ \forall y \in \Omega$,} \tag{$EP1$}
\end{equation}
where $\mu = (\mu_1, \ldots, \mu_n)^T, \alpha = (\alpha_1, \ldots, \alpha_n)^T$, and
\beqs
\tilde{B} = \begin{bmatrix}
0 & \delta_1 & \delta_1 & \ldots & \delta_1\\
\delta_2 & 0 & \delta_2 & \ldots & \delta_2\\
\cdot & \cdot & \cdot & \ldots & \cdot\\
\delta_n & \delta_n & \delta_n & \ldots & 0
\end{bmatrix},\qquad
B =  \begin{bmatrix}
2 \delta_1 & 0 & 0 & \ldots & 0\\
0 & 2 \delta_2 & 0 & \ldots & 0\\
\cdot & \cdot & \cdot & \ldots & \cdot\\
0 & 0 & 0 & \ldots & 2 \delta_n
\end{bmatrix}.
\eeqs
Note that $f(x, y) = f_1(x, y) + f_2(x, y)$ in which
\begin{align*}
f_1(x, y) &= (\tilde{B} x + \mu - \alpha)^T (y - x),\\
f_2(x, y) &= \frac{1}{2} y^T B y - \frac{1}{2} x^T B x.
\end{align*}
It is obvious that $f, f_1, f_2$ are equilibrium functions satisfying conditions (A1)-(A3).

For numerical experiments, we set $C_i = [10, 50]$ for $i = 1, \ldots, n$,
$\underline{\sigma} = 10n + 10$, and $\overline{\sigma} = 50n - 10$.
The initial guess was set to $x^0_i = 30 (i = 1, \ldots, n)$.
We tested the algorithm on problem instances
with different numbers $n$ of companies but having the following fixed
values of parameters $\alpha_i = 120, \delta_i = 1, \mu_i = 30$ for $i = 1, \ldots, n$.
Table \ref{TableExp2} reports the outcomes of Algorithm 1 with restart strategy applied to
these instances for different values of dimension $n$ and appropriate values of parameters
$\beta_k$.

\begin{table}[H]
\centering
\begin{tabular}{|c|c|c||c|c|}
\hline
\multirow{3}{*}{$n$} & \multirow{3}{*}{$\beta_k$} & \multirow{2}{*}{Total number of} & \multirow{2}{*}{Number of} & \multirow{2}{*}{Number of iterations}\\
 & & & & \\
 & & iterations & restarts & from the last restart\\
\hline
2 & $10/(k+1)$ & 2 & 0 & 2\\
3 & $10/(k+1)$ & 639 & 2 & 9\\
4 & $10/(k+1)$ & 911 & 2 & 4\\
5 & $10/(k+1)$ & 1027 & 2 & 2\\
10 & $10/(k+1)$ & 1201 & 1 & 2\\
10 & $100/(k+1)$ & 266 & 1 & 2\\
15 & $10/(k+1)$ & 2967 & 2 & 2\\
15 & $100/(k+1)$ & 408 & 1 & 2\\
20 & $10/(k+1)$ & 5007 & 2 & 2\\
20 & $100/(k+1)$ & 539 & 1 & 2\\
\hline
\end{tabular}
\caption{Performance of Algorithm 1 in solving linear Cournot oligopolistic model
with additional joint constraints.}
\label{TableExp2}
\end{table}

On one hand, the results reported in Table \ref{TableExp2} show the applicability
of Algorithm 1 for solving linear Cournot-Nash oligopolistic model with joint constraints.
On the other hand, it follows from this table that the choice of parameter $\beta_k$ is
crucial for the convergence of the algorithm, since changing the value of this parameter
may significantly reduce the number of iterations. Furthermore, the last two columns
of Table \ref{TableExp2} show that, by applying our suggested restart strategy,
we can find `good' starting points from that the algorithm terminated after few iterations.


\section{Conclusion}\label{SectionConclusion}

We have proposed splitting algorithms for monotone equilibrium problems where
the bifunction is the sum of the two ones. The first algorithm uses an ergodic sequence ensuring
convergence without extragradient (double projection). The second one is for paramonotone
equilibrium problems ensuring convergence without using the ergodic sequence.
A restart strategy has been used to enhance the convergence of the proposed algorithms.


\end{document}